\documentclass[11pt,a4paper]{article}
\usepackage{amsbsy,amsfonts,amsgen,amsmath,amsopn,amstext,amsthm,
	cases,color,
	epsf,epsfig,ebezier,eepic,epstopdf,
	listings,
	multirow
}
\usepackage{graphicx}
\usepackage{subcaption}
\usepackage{pgf,tikz}
\usepackage[marginal]{footmisc}
\usepackage{enumitem}
\usepackage[titletoc]{appendix}
\usepackage{booktabs}
\usepackage{url}
\usepackage{mathtools}

\usepackage{pgfplots}
\usepackage{authblk}
\usepackage{amssymb}

\usepackage{empheq}

\usepackage{dsfont}
\pgfplotsset{compat=1.18}
\usepackage{mathrsfs}
\usepackage{wasysym}
\usetikzlibrary{arrows}
\usepackage{aligned-overset}

\usepackage{algorithm}
\usepackage{algpseudocode}

\usepackage{algcompatible}
 \usepackage[backref=page]{hyperref}
\allowdisplaybreaks[1]

\definecolor{uuuuuu}{rgb}{0.27,0.27,0.27}
\definecolor{sqsqsq}{rgb}{0.1255,0.1255,0.1255}

\newtheorem{definition}{Definition} [section]
\newtheorem{theorem}[definition]{Theorem}
\newtheorem{lemma}[definition]{Lemma}

\newtheorem{conjecture}[definition]{Conjecture}
\newtheorem{claim}[definition]{Claim}

\newtheorem{fact}[definition]{Fact}
\newenvironment{pf}{\noindent{\bf Proof}}

\newcommand{\uproduct}{\mathbin{\;{\rotatebox{90}{\textnormal{$\small\Bowtie$}}}}}

\newcommand{\etal}{\textit{et al}. }

\def\qed{\hfill \rule{4pt}{7pt}}

\begin{document}
\title{\bf\Large Triangles in graphs without the expansion of $4$-cycle}

\date{}
%%%%%%%%%%%%%%%%%%%%%%%%%%%%%%%%%%%%%%%%%%%%%%%%%%%%
% \author[1]{Xizhi Liu\thanks{Research was supported by ERC Advanced Grant 101020255. Email: \texttt{xizhi.liu.ac@gmail.com}}}
\author{Jialei Song
%\thanks{Research was supported by Science and Technology Commission of Shanghai Municipality (No. 22DZ2229014) and National Natural Science Foundation of China (No. 12331014, No. 12501481, No. 12271169).} 
\footnote{Email: \texttt{jlsong@math.ecnu.edu.cn or jialsong@foxmail.com.}}, Qi Wu \footnote{Email: \texttt{wuqimath@163.com.}}, Long-Tu Yuan\footnote{Email: \texttt{ltyuan@math.ecnu.edu.cn(Corresponding author).}}}

%%%%%%%%%%%%%%%%%%%%%%%%%%%%%%%%%%%%%%%%%%%%%%%%%%%%%
% \affil[1]{Mathematics Institute and DIMAP,
%             University of Warwick,
%             Coventry, CV4 7AL, UK}
\affil{
School of Mathematical Sciences,  Key Laboratory of MEA (Ministry of Education) \& Shanghai Key Laboratory of PMMP,  East China Normal University, Shanghai 200241, China
}
%%%%%%%%%%%%%%%%%%%%%%%%%%%%%%%%%%%%%%%%%%%%%%%%%%%
\maketitle
%\footnote{footnote}
%%%%%%%%%%%%%%%%%%%%%%%%%%%%%%%%%%%%%%%%%%%%%%%%%
\begin{abstract}
The expansion $F^{\triangle}$ of a graph $F$ is the graph obtained from $F$ by replacing each edge with a triangle.
Lv \etal proposed a conjecture on the maximum number of triangles in a graph without $P_k^{\triangle}$ or $C_k^{\triangle}$ for every $k \ge 4$.
Their conjecture was confirmed in previous work for $P_k^{\triangle}$ when $k \ge 4$ and $C_k^{\triangle}$ when $k \ge 5$.
In this note, we resolve the remaining case $C_4^{\triangle}$, demonstrating that this is the only counterexample to their conjecture.

% The $\triangle$-blowup $G^{\triangle}$ of $G$ is the graph obtained from $F$ by replacing each edge with a copy of $K_3$ such that no two different copies of $K_3$ use the same new vertex.
% For large $n$ and $k\ge 4$, Lv \etal%~\cite{LGHSTVZ24}
% proposed a conjecture on the generalized Tur\'{a}n numbers of $P_k^{\triangle}$ and $C_k^{\triangle}$.
% Almost all case, except $C_4^{\triangle}$, in this conjecture were confirmed by Liu and Song.
% In this paper, we give an answer to the generalized Tur\'{a}n numbers of $C_4^{\triangle}$, which negated the conjecture proposed by Lv \etal.

\medskip


\end{abstract}
%%%%%%%%%%%%%%%%%%%%%%%%%%%%%%%%%%%%%%%%%%%%%%%%%
\section{Introduction}\label{SEC:Introduction}
We identify a graph $G$ with its edge set and use $V(G)$ to denote its vertex set.
% Given a graph $G$, let $V(G)$ denote its vertex set and $E(G)$ its edge set.
For every  $k>0$, let $C_k$, $P_k$, $M_k$ denote the \textbf{cycle}, \textbf{path} and \textbf{matching} with $k$ edges, respectively.
For simplicity, we use the vertex sequence $v_1v_2\cdots v_kv_1$  to denote a $C_k$, where $V(C_k)=\{v_1,v_2,\cdots, v_k\}$ and $v_iv_{i+1}\in E(C_k)$ for $1\le i\le k$ under taking index modulo $k$.
Let $K_t$ represent the complete graph with $t$ vertices, and $K_{s,t}$ be the complete bipartite graph with parts of sizes $s$ and $t$, respectively.
The \textbf{expansion} $G^{\triangle}$ of $G$ is the graph obtained from $G$ by replacing each edge with a copy of $K_3$, ensuring that no two different copies of $K_3$ share the same new vertex (see Figure~\ref{fig:Path-Cycle-Expansion} for $C_4^{\triangle}$).
%
%Denote by $G\cup H$ the vertex-disjoint union of two graphs $G$ and $H$. Given two graphs $G_1$ and $G_2$, the \textbf{join} $G_1 \uproduct G_2$ is obtained by taking vertex-disjoint copies of $G_1$ and $G_2$ and adding all pairs that have nonempty intersection with both $V(G_1)$ and $V(G_2)$, \ie
Given two graphs $G_1$ and $G_2$, the \textbf{join} $G_1 \uproduct G_2$ is the graph with vertex set $V(G_1)\cup V(G_2)$ and edge set 
$\left\{uv \colon u\in V(G_1),\ v\in V(G_2)\right\} \cup E(G_1) \cup E(G_2).$
Let $T(n)$ be a balanced complete bipartite graph and let $X,Y$ be two parts of $T(n)$ with $|X|\le |Y|$.
We define $T^+(n)$ as the graph obtained from $T(n)$ by adding an edge into $X$.
For a positive integer $t$, define $S(n,t):=K_t\uproduct T(n-t)$ and $S^+(n,t):=K_t\uproduct T^+(n-t)$.

Let $Q, F$ be two graphs.
We denote the number of copies of $Q$ in $F$ as $N(Q, F)$.
The \textbf{generalized Tur\'{a}n number} $\mathrm{ex}(n,Q,F)$ is the maximum number of copies of $Q$ in an $n$-vertex $F$-free graph.
When $Q = K_2$, this corresponds to the ordinary Tur\'{a}n number of $F$.
In 1941, Tur\'{a}n~\cite{T41} finished a seminal work in extremal graph theory by showing the unique extremal graph for clique.
Comparable to the conventional Tur\'{a}n problem, Zykov~\cite{Zykov} firstly considered the generalized Tur\'{a}n problem, in which he determined $\mathrm{ex}(n,K_s, K_t)$ for all $n\ge t > s \ge 2$.
Later, Erd\H{o}s~\cite{E62} proved a similar result.
%Erd\H{o}s considered the generalized Tur\'{a}n problem in~\cite{E62}, in which he determined $\mathrm{ex}(n,K_s, K_t)$ for all $t > s \ge 3$.
In~\cite{AS16}, Alon and Shikhelman started a systematic study of $\mathrm{ex}(n,Q,F)$, and since then this topic has attracted much attention.
Especially, determining the generalized Tur\'{a}n number $\mathrm{ex}(n,K_3,F)$ for various graphs $F$ became worth-thinking.
In~\cite{GP2022}, Gerbner and Patk{\'o}s determined $\mathrm{ex}(n,K_3, S_p^{\triangle})$ and $\mathrm{ex}(n,K_3, M_p^{\triangle})$ for any $p$ and sufficiently large $n$.
Later, Yuan and Yang~\cite{Yuan2022} determined $\mathrm{ex}(n,K_3, M_2^{\triangle})$ for all $n$, and Zhu \etal~\cite{Zhu2023} determined $\mathrm{ex}(n,K_3, S_k^{\triangle})$ for positive $k$ and $n\ge 4k^3$.
Recently, Lv \etal~\cite{LGHSTVZ24} determined the exact value of $\mathrm{ex}(n,K_3,P_3^{\triangle})$ for $n\ge 300^{300}$, and $\mathrm{ex}(n,K_3,C_3^{\triangle})$ for $n\ge 22$.
In addition, they proposed the following conjecture:
\begin{conjecture}[{\cite{LGHSTVZ24}}]\label{conj:Lv-etc}
    For $k\ge 4$ and sufficiently large $n$,
    \begin{align*}
            \mathrm{ex}(n,K_3,P_k^{\triangle})=\mathrm{ex}(n,K_3,C_k^{\triangle})=
                \begin{cases}
                N\left(K_3 ,S(n,\lfloor \frac{k-1}{2}\rfloor) \right), &  \text{if  $k$ is odd}, \\
                N\left(K_3 ,S^+(n,\lfloor \frac{k-1}{2}\rfloor) \right), &  \text{if  $k$ is even}.
                \end{cases}
             \end{align*}
\end{conjecture}
By combining the stability method with methods from Extremal Set Theory, specifically those employed by Kostochka, Mubyai and Verstra\"{e}te~\cite{KMV15a,KMV17b} for hypergraph linear cycles and paths, Liu, Song and Yuan~\cite{LS23a} confirmed Conjecture~\ref{conj:Lv-etc} for $P_k^{\triangle}$ when $k \ge 4$ and for $C_k^{\triangle}$ when $k \ge 5$.
In this note, we determine the value of $\mathrm{ex}(n,K_3,C_4^{\triangle})$ for sufficiently large $n$.
In particular, our result shows that Conjecture~\ref{conj:Lv-etc} is false for $C_{4}^{\triangle}$.

Define $T^=(n)$ as the graph obtained from $T(n)$ by adding $e_1$ into $X$, $e_2$ into $Y$ and subsequently removing a copy of $M_2$ between $e_1$ and $e_2$ in $T(n)$.
Let $S(n):=S(n,1), S^+(n):=S^+(n,1)$ and $S^=(n):=K_1\uproduct T^=(n-1)$ (view $S^=(n)$ in Figure~\ref{fig:Path-Cycle-Expansion}).
It is easy to verify that any $C_4$ in $S^=(n)$ can not expand into a copy of $C_4^{\triangle}$ in $S^=(n)$. In fact, the edges between $X$ and $Y$  can only form a triangle with $x_0$ in $S^=(n)$. So if $C_4$ in $C_{4}^{\triangle}$ contains both or none of $e_1$ and $e_2$, then there are at least two edges of $C_4$ between $X$ and $Y$, a contradiction to the fact. If $C_4$ in $C_{4}^{\triangle}$ contains only one of $e_1$ and $e_2$, then $C_4$ contains vertex $x_0$ and an edge between $X$ and $Y$, a contradiction to the fact.  Hence, $S^=(n)$ is $C_4^{\triangle}$-free.
Moreover, for $n\ge 7$, a simple calculation shows that $N(K_3 ,S^+(n))  = N(K_3 ,S(n))+\lceil \frac{n-1}{2}\rceil$ and
\begin{align*}
	N\left(K_3 ,S^=(n)\right) & = N\left(K_3 ,S(n)\right) -2 +\left( \left\lceil \frac{n-1}{2} \right\rceil -2 \right) +\left( \left\lfloor \frac{n-1}{2} \right\rfloor -2 \right) \\
	& = N\left(K_3 ,S(n) \right) + (n - 7) .
\end{align*}
Note that $S(n)$ is a proper subgraph of $S^+(n)$. Therefore, it is evident that
\begin{equation}\label{equ:sequence S}
N\left(K_3 ,S^=(n) \right) > N\left(K_3 ,S^+(n) \right) > N\left(K_3 ,S(n) \right) \text{ for } n\ge 15.
\end{equation}
%In particular, $N\left(K_3 ,S^=(n)\right) = N\left(K_3 ,S(n) \right) + (n - 7).$
Our main theorem is as follows:
\begin{theorem}\label{THM:Main}
    Suppose that $n$ is sufficiently large.
    Then $\mathrm{ex}(n, K_{3}, C_4^{\triangle}) = N\left(K_3 ,S^=(n)\right)$.
    Moreover, $S^=(n)$ is the unique extremal graph.
\end{theorem}
%
%%%%%%%%%%%%%%%%%%%%%%%%%%%%%%%%%%%%%%%%%%%
%%%%%%%%%%%%%%%%%%%%%%%%%%%%%%%%%%%%%%
\begin{figure*}[htbp]
\centering
\begin{subfigure}{0.4\textwidth}
\centering
\tikzset{every picture/.style={line width=0.75pt}} %set default line width to 0.75pt
\begin{tikzpicture}[x=0.75pt,y=0.75pt,yscale=-1,xscale=1,scale=0.8]

\draw[line width=1.5pt,color=sqsqsq]  (0,0) -- (0,60) -- (60,60) -- (60,0) -- cycle;

\draw[line width=1pt]  (0,0) -- (-40,30) -- (0,60) -- cycle ;
\draw[line width=1pt]  (0,0) -- (30,-40) -- (60,0) -- cycle ;
\draw[line width=1pt]  (60,60) -- (30,100) -- (0,60) -- cycle ;
\draw[line width=1pt]  (60,60) -- (100,30) -- (60,0) -- cycle ;

\draw [fill=uuuuuu] (0,0) circle (2pt);
\draw [fill=uuuuuu] (0,60) circle (2pt);
\draw [fill=uuuuuu] (60,60) circle (2pt);
\draw [fill=uuuuuu] (60,0) circle (2pt);
\draw [fill=blue]   (-40,30)  circle (2pt);
\draw [fill=blue]   (30,-40) circle (2pt);
\draw [fill=blue]   (30,100)  circle (2pt);
\draw [fill=blue]   (100,30)  circle (2pt);
\draw [fill=blue]   (0,130)  circle (0pt);
\end{tikzpicture}
%\caption{The configuration of $C_4^{\triangle}$.}
%\label{}
\end{subfigure}
\quad
\begin{subfigure}{0.4\textwidth}
\centering
\tikzset{every picture/.style={line width=0.75pt}} %set default line width to 0.75pt
\begin{tikzpicture}[x=0.75pt,y=0.75pt,yscale=-1,xscale=1,scale=0.8]

\draw[line width=1pt] (-60,-40) ellipse (30 and 100) ;
\draw[line width=0.8pt, fill=sqsqsq]  (-60,-90) -- (-60,-120) ;
\node at (-70,-105){\footnotesize $e_1$};

\draw[line width=1pt] (60,-40) ellipse (30 and 100) ;
\draw[line width=0.8pt, fill=sqsqsq]  (60,-90) -- (60,-120) ;
\node at (70,-105){\footnotesize $e_2$};

\draw[line width=0.8pt, fill=sqsqsq]  (-60,-90) --  (60,-90) ;
\draw[line width=0.8pt, dashed]  (-60,-90) --  (60,-120) ;
\draw[line width=0.8pt, dashed]  (-60,-120) --  (60,-90) ;
\draw[line width=0.8pt, fill=sqsqsq]  (-60,-120) --  (60,-120) ;

\draw[line width=0.5pt, fill=sqsqsq]  (0,-140) --  (-60,-90) ;
\draw[line width=0.5pt, fill=sqsqsq]  (0,-140) -- (-60,-120) ;
\draw[line width=0.5pt, fill=sqsqsq]  (0,-140) --  (-60,-40) ;
\draw[line width=0.5pt, fill=sqsqsq]  (0,-140) -- (-60,-10) ;
\draw[line width=0.5pt, fill=sqsqsq]  (0,-140) --  (-60,20) ;
\draw[line width=0.5pt, fill=sqsqsq]  (0,-140) --  (60,-90) ;
\draw[line width=0.5pt, fill=sqsqsq]  (0,-140) -- (60,-120) ;
\draw[line width=0.5pt, fill=sqsqsq]  (0,-140) --  (60,-40) ;
\draw[line width=0.5pt, fill=sqsqsq]  (0,-140) -- (60,-10) ;
\draw[line width=0.5pt, fill=sqsqsq]  (0,-140) --  (60,20) ;

\draw[line width=0.5pt, fill=sqsqsq]  (60,-90) --  (-60,-40) ;
\draw[line width=0.5pt, fill=sqsqsq]  (60,-90) -- (-60,-10) ;
\draw[line width=0.5pt, fill=sqsqsq]  (60,-90) --  (-60,20) ;
\draw[line width=0.5pt, fill=sqsqsq]  (60,-120) --  (-60,-40) ;
\draw[line width=0.5pt, fill=sqsqsq]  (60,-120) -- (-60,-10) ;
\draw[line width=0.5pt, fill=sqsqsq]  (60,-120) --  (-60,20) ;

\draw[line width=0.5pt, fill=sqsqsq]  (-60,-90) --  (60,-40) ;
\draw[line width=0.5pt, fill=sqsqsq]  (-60,-90) -- (60,-10) ;
\draw[line width=0.5pt, fill=sqsqsq]  (-60,-90) --  (60,20) ;
\draw[line width=0.5pt, fill=sqsqsq]  (-60,-120) --  (60,-40) ;
\draw[line width=0.5pt, fill=sqsqsq]  (-60,-120) -- (60,-10) ;
\draw[line width=0.5pt, fill=sqsqsq]  (-60,-120) --  (60,20) ;

\draw[line width=0.5pt, fill=sqsqsq]  (60,-40) --  (-60,-40) ;
\draw[line width=0.5pt, fill=sqsqsq]  (60,-40) -- (-60,-10) ;
\draw[line width=0.5pt, fill=sqsqsq]  (60,-40) --  (-60,20) ;
\draw[line width=0.5pt, fill=sqsqsq]  (60,-10) --  (-60,-40) ;
\draw[line width=0.5pt, fill=sqsqsq]  (60,-10) -- (-60,-10) ;
\draw[line width=0.5pt, fill=sqsqsq]  (60,-10) --  (-60,20) ;
\draw[line width=0.5pt, fill=sqsqsq]  (60,20) --  (-60,-40) ;
\draw[line width=0.5pt, fill=sqsqsq]  (60,20) -- (-60,-10) ;
\draw[line width=0.5pt, fill=sqsqsq]  (60,20) --  (-60,20) ;

\draw [fill=blue]   (-60,-40) circle (2pt);
\draw [fill=blue]   (-60,-10) circle (2pt);
\draw [fill=blue]   (-60,20) circle (2pt);
\draw [fill=blue]   (-60,-90) circle (2pt);
\draw [fill=blue]   (-60,-120) circle (2pt);

\draw [fill=blue]   (60,-40) circle (2pt);
\draw [fill=blue]   (60,-10) circle (2pt);
\draw [fill=blue]   (60,20) circle (2pt);
\draw [fill=blue]   (60,-90) circle (2pt);
\draw [fill=blue]   (60,-120) circle (2pt);

\node at (0,-150){\footnotesize $x_0$};
\draw [fill=uuuuuu] (0,-140) circle (3pt);

\end{tikzpicture}
%\caption{The configuration of $S^=(n)$.}
%\label{}
\end{subfigure}
\caption{The graph $C_4^{\triangle}$ and the graph $S^=(n)$.}
\label{fig:Path-Cycle-Expansion}
\end{figure*}
%uncomment if require: \path (0,254); %set diagram left start at 0, and has height of 254
%%%%%%%%%%%%%%%%%%%%%%%%%%%%%%%%%%%%%%

%The remainder of this paper is organized as follows.
%
%In Section~\ref{SEC:Preliminaries}, we present some definitions and preliminary results related to $\triangle$-blowup graphs.
%
%In Section~\ref{SEC:Thm-Main}, we prove our main theorem.
%
%%%%%%%%%%%%%%%%%%%%%%%%%%%%%%%%%%%%%%%%%%%%%%%%
\section{The proof of Theorem~\ref{THM:Main}}\label{SEC:Thm-Main}
%
%Using the stability method developed by Siminovits~\cite{S68}, the proof of Theorem~\ref{THM:Main} follows a route similar to the one used in the proof of~{\cite[Theorem 1.9]{LS23a}}, differing only in a few places, such as Claims~\ref{CLAIM:GenTuran-cycle-B1},~\ref{CLAIM:GenTuran-cycle-GD1'D2'} and~\ref{CLAIM:GenTuran-even-cycle-B0}.
%Thus, we move proofs for repetitive parts to the appendix for interested readers.
Before moving on to the proof of main theorem, we introduce some necessary definitions.

Let $X$ and $Y$ be two disjoint vertex sets in $G$.
We use $G[X]$ to denote the induced subgraph on $X$,
$G[X,Y]$ to denote the induced bipartite graph between $X$ and $Y$ in $G$,
$|G[X,Y]|$ to denote the number of edges between $X$ and $Y$.
Choose $v\in V(G)$, we define {\bf links} of $v$ in $X$ by
\begin{align*}
    L_{X}(v)\coloneqq \{e\in E(G[X]) \colon \text{there exists a triangle composed by $v$ and $e$ in $G$}\},
\end{align*}
and for an edge $e\in E(G[X])$,
\begin{align*}
    N_X^{\ast}(e) \coloneqq \{u\in X\setminus V(e) \colon  \text{there exists a triangle  in $G$ containing both $e$ and $u$}\}
\end{align*}
as {\bf co-neighborhood} of $e$ in $X$.
Moreover, let $d_X(e)=|N_X^{\ast}(e)|$.

We characterize ${F}$ as a \textbf{$2$-intersecting graph} if it consists exclusively of triangles that share precisely one edge with each other.
Furthermore,
\begin{equation}\label{2-intersecting}
	\mbox{a $2$-intersecting graph on $t$ vertices contains exactly $t-2$ triangles.}
\end{equation}
Given a graph $G$, define an associated $3$-graph as follows:
$\mathcal{K}_{G} \coloneqq \left\{e \in \binom{V(G)}{3} \colon G[e] \cong K_3\right\}.$
Pertinent to this, we can confirm:
\begin{fact}\label{FACT:shadow-Turan}
    For every $F^{\triangle}$-free graph $G$, the $3$-graph $\mathcal{K}_{G}$ is $\mathcal{K}_{F^{\triangle}}$-free.
\end{fact}
Let $\mathrm{ex}_3(n, \mathcal{K}_{G})$ be the maximum number of $3$-edges in a $\mathcal{K}_{G}$-free 3-graph on $n$-vertices.
In~\cite{KMV15a}, Kostochka, Mubyai and Verstra\"{e}te gave an asymptotic result on $\mathrm{ex}_3(n, \mathcal{K}_{C_k})$ as follows:
\begin{lemma}[{\cite[Theorem~6.1]{KMV15a}}]\label{KMV asymptotics}
	Let $k\ge 4$, $\ell=\left\lfloor\frac{k-1}{2}\right\rfloor$ and $n$ be sufficiently large.
	Then $$\mathrm{ex}_3(n, \mathcal{K}_{C_k}) = \ell \binom{n}{\ell-1} + o(n^{2}).$$
\end{lemma}
Drawing from Fact~\ref{FACT:shadow-Turan} and Lemma~\ref{KMV asymptotics},
we identify a positive constant $C$ satisfying
\begin{equation}\label{equ:upper bound Cktriangle}
    \mathrm{ex}(n,K_3,C_4^{\triangle})\le \mathrm{ex}_3(n, \mathcal{K}_{C_4}) \le Cn^2.
\end{equation}

To prove our main result, a key tool is the following stability lemma.
\begin{lemma}[{\cite[Theorem~1.11]{LS23a}}]\label{THM:GenTuran-Cycle-Stability}
    For every positive constant $C^{-1}\gg \delta >0$, there exist $\varepsilon>0$ and $n_0\in \mathbb{N}$ such that the following holds for all $n \ge n_0$.
    Suppose that $G$ is an $n$-vertex $C_4^{\triangle}$-free graph with $N(K_{3}, G) \ge\left(\frac{1}{4} - \varepsilon\right)n^2,$
    then there exists a vertex $x$ such that
\begin{enumerate}[label=(\roman*)]

    \item $|E(G-x)| \ge n^2/4 - \delta n^2$, and
    \item $G-x$ can be made bipartite by removing at most $\delta n^2$ edges, and
    \item $N(K_3,G-x)\leq \delta n^{2}$.

\end{enumerate}
\end{lemma}
Let $\delta$ be sufficiently small, $C$ and $\varepsilon$ be the constants defined in Lemma~\ref{THM:GenTuran-Cycle-Stability}, and $n\gg C$ be sufficiently large such that $N(K_3, S(n)) > \left(\frac{1}{4} - \varepsilon\right)n^2$.
Let $G$ be an $n$-vertex $C_4^{\triangle}$-free graph with $N(K_3, G)\geq N(K_3, S^=(n)) > \left(\frac{1}{4} - \varepsilon\right)n^2$.
We may assume that every edge in $G$ is contained in some triangle of $G$, as otherwise deleting it from $G$ will not impact the value of $N(K_3, G)$.
Clearly, graph $G$ satisfies
statements $(i), (ii)$ in Lemma~\ref{THM:GenTuran-Cycle-Stability}.
%statement $(i)\sim(iv)$ in Lemma~\ref{THM:GenTuran-Cycle-Stability}.

In the subsequent discussion, we maintain the designation of the vertex $x$.
Let $V \coloneqq V(G)$, $V'\coloneqq V\setminus \{x\}$ and $V_1 \cup V_2 $ be a bipartition of $V'$ such that $|G[V_1,V_2]|$ is maximized.
We contruct an auxiliary graph $S:= \{x\}\uproduct \left(V_1 \uproduct V_2\right)$ on $V(G)$ with $x$ indentical.
Let $\mathcal{G}$ be the set of triangles in $G$,
$\mathcal{S}$ be the set of triangles in $S$,
$\mathcal{B} \coloneqq \mathcal{G}\setminus \mathcal{S}$ and $\mathcal{M} \coloneqq \mathcal{S}\setminus \mathcal{G}$.
%Define $S$, $B$ as the graphs spanned by edges in triangles of $\mathcal{S}$, $\mathcal{B}$, respectively.
Our aim is to prove $G \cong S^=(n)$ by showing $|{\cal G}| = |{\cal S}\cup {\cal B}\setminus {\cal M}| \le N\left(K_3 ,S^=(n) \right)$.

Considering statements~$(i),(ii)$ in Lemma~\ref{THM:GenTuran-Cycle-Stability}, we deduce that
    \begin{align}\label{equ:GenTuran-cycle-G[V1,V2]}
        |G[V_1, V_2]| \ge \frac{n^2}{4} - 2\delta n^2,
    \end{align}
where $\delta$ is sufficiently small.
With some simple calculations, inequality~\eqref{equ:GenTuran-cycle-G[V1,V2]} implies that
    \begin{align}\label{equ:GenTuran-cycle-Vi-size}
        \left(\frac{1}{2} - \sqrt{2\delta}\right)n
            \le |V_i|
            \le \left(\frac{1}{2} + \sqrt{2\delta}\right)n,
            \text{ for } i\in \{1,2\}.
    \end{align}
Let $\tau \coloneqq n/200, D \coloneqq \left\{v\in V' \colon d_{V(G)}(v x) \ge \tau\right\}, \overline{D} \coloneqq V'\setminus D, D_i  \coloneqq D \cap V_i$ and $\overline{D}_i \coloneqq V_i \setminus D_i$ for $i\in \{1,2\}$.
Partition each $D_i$  into two additional subsets as follows:
Define $\tau' \coloneqq {2n}/{5}$, $D_i' \coloneqq \left\{v\in D_i \colon |N_{G}(v)\cap V_{3-i}| \ge \tau'\right\}$ and $\overline{D'}_i \coloneqq D_i\setminus D_i'$ for $i\in \{1,2\}$.
Set $D'=D_1'\cup D_2'$ and $\overline{D'} \coloneqq D\setminus D'$.
We further divide the triangles of $\mathcal{M}$ by letting $\mathcal{M}_1 \coloneqq \{K_3\in \mathcal{M} \colon$triangle $K_3$ intersects $\overline{D}$ nonempty$\}$ and $\mathcal{M}_2 \coloneqq \mathcal{M}\setminus \mathcal{M}_1$.
%We will establish  Claim~\ref{CLAIM:GenTuran-cycle-D-bar}$\sim$\ref{CLAIM:GenTuran-cycle-B_2'} based on Claim $5.3, 5.4(i), 5.5, 5.6, 5.8, 5.10$ in~\cite{LS23a}, respectively.

Claims~\ref{CLAIM:GenTuran-cycle-D-bar}$\sim$\ref{CLAIM:GenTuran-cycle-B_2'} arise from Claims $5.3, 5.4(i), 5.5, 5.6, 5.8, 5.10$ in~\cite{LS23a}  respectively.
For completeness, we prove them here in a simpler form.

    \begin{claim}[{\cite[Claim~5.3]{LS23a}}]\label{CLAIM:GenTuran-cycle-D-bar}
        We have $|\mathcal{M}_1| \ge 49 n|\overline{D}|/100$ and  $|\overline{D}| \le 8 \delta n$.
        %We have $|\overline{D}| \le 8 \delta n$.
    \end{claim}
    \begin{proof}
    	Since $|L_{{G}}(x) \cap L_{S}(x)| = |L_{{G}}(x) \cap G[V_1,V_2]| \ge {n^2}/{4} - 3\delta n^2$, by the definition of $\mathcal{M}$,
    	\begin{align}\label{equ:GenTuran-cycle-m-upper}
    		|\mathcal{M}|
    		\le |L_{S}(x)|-|L_{{G}}(x) \cap L_{S}(x)|
    		\le |V_1|\times |V_2|- \left( \frac{n^2}{4} - 3\delta n^2 \right)
    		\le 3\delta n^2.
    	\end{align}
    	From the definition of $D$, every vertex $v\in \overline{D}_i$ satisfies $|N_{{V(G)}}^{\ast}(vx) \cap V_{3-i}| \le \tau$, for $i\in \{1,2\}$.
    	Hence, the edge $vx$ is contained in at least $|V_{3-i}| - \tau$ triangles in $\mathcal{M}_1$.
    	Therefore, it follows from~\eqref{equ:GenTuran-cycle-Vi-size} that
    	\begin{align*}
    		|\mathcal{M}_1|
    		\ge \left(\min\{|V_1|, |V_2|\} - \tau \right) \left(|\overline{D}_1| + |\overline{D}_2|\right)
    		\ge \left(\frac{n}{2}-\sqrt{2\delta}n - \tau\right)|\overline{D}|
    		\ge \frac{49}{100} n |\overline{D}|.
    	\end{align*}
    	Combined with~\eqref{equ:GenTuran-cycle-m-upper}, we obtain $|\overline{D}|
    	\le 3 \delta n^2 \times 100 / (49n)
    	\le 8 \delta n$ (as $\mathcal{M}_1 \subseteq {\cal M}$ and $\delta$ is small enough).
    \end{proof}
    From~\eqref{equ:GenTuran-cycle-Vi-size} and Claim~\ref{CLAIM:GenTuran-cycle-D-bar}, we can obverse that % for sufficiently small $\delta$,
    \begin{equation}\label{equ:GenTuran-Di-size}
        |D_i|\ge \left(\frac{1}{2} - \sqrt{2\delta}\right)n - 8 \delta n \ge  \left(\frac{1}{2} - 2\sqrt{\delta}\right)n.
    \end{equation}
    \begin{claim}[{\cite[Claim~5.4~(i)]{LS23a}}]\label{CLAIM:GenTuran-cycle-P_2}
        For any $v\in V'$, there will not exist two disjoint edges $e_1,e_2\in L_G(v)\cap E(G[V'])$ such that both $e_1,e_2$ intersect $D$ nonempty.
    \end{claim}
    \begin{proof}
    	We first introduce a fact, which can be proved via greedy algorithm and presents a sufficient condition for the expansion from a certain graph.
    	\begin{fact}\label{FACT:partial-embedding}
    		Let $F$ be a subgraph of $G$ with $m\ge 1$ edges
    		and $e_1 e_2 \dots e_m$ be an edge sequence of $F$.
    		If there exists $1\le t \le m$ such that
    		$$\mbox{ there exist distinct co-neighbors for $e_1, \dots, e_t$, and $d_{G}(e_j) \ge 3m$ for all $t< j \le m$, }$$
    		then $F^{\triangle}$ is also a subgraph of $G$.
    	\end{fact}
        Now, we come to the proof of Claim~\ref{CLAIM:GenTuran-cycle-P_2}.

        Suppose to the contrary that there exists $v\in V'$ and two edges $e_1,e_2\in L_G(v)\cap E(V')$ such that $v_i\in V(e_i)\cap D$ for $i=1,2$.
    	Notice that $xv_1vv_2x$ is a $4$ cycle.
    	By definition of $D$ and Fact~\ref{FACT:partial-embedding}, we can expand $xv_1vv_2x$ into a copy of $C_{4}^{\triangle}$~(the edge sequence $v_1v,vv_2,v_2x,xv_1$, with $t=2$ in Fact~\ref{FACT:partial-embedding}).
    \end{proof}
    \begin{claim}[{\cite[Claim~5.5]{LS23a}}]\label{CLAIM:GenTuran-cycle-Maxdeg}
        For every $v\in V'$,
        $%\begin{align*}
            \min\left\{|N_{G}(v) \cap D_1|, |N_{G}(v) \cap D_2|\right\}
            \le \sqrt{3\delta}n.
        $%\end{align*}
    \end{claim}
    \begin{proof}
    	Suppose to the contrary that there exists a vertex $v\in V'$ such that the set $N_i \coloneqq N_{G}(v) \cap D_i$ has size at least $\sqrt{3\delta}n$ for both $i \in \{1,2\}$.
    	Observe from Claim~\ref{CLAIM:GenTuran-cycle-P_2} that the induced bipartite graph $G[N_1, N_2]$ does not contain two disjoint edges (as otherwise those two disjoint edges together with $v$ will contradict Claim~\ref{CLAIM:GenTuran-cycle-P_2}).
    	Hence, $|G[N_1, N_2]| \le n$ and consequently,
    	$|G[V_1, V_2]| \le |V_1||V_2| - \left(|N_1||N_2| - |G[N_1, N_2]|\right)
    	\le {n^2}/{4} - (3\delta n^2 - n)
    	< {n^2}/{4} - 2\delta n^2,$
    	contradicting~\eqref{equ:GenTuran-cycle-G[V1,V2]}.
    \end{proof}
    \begin{claim}[{\cite[Claim~5.6]{LS23a}}]\label{CLAIM:GenTuran-bad-deg-max}
        For every $i\in \{1,2\}$ and $v\in V_i$,
        $%\begin{align*}
            |N_{G}(v) \cap D_i| \le 11\sqrt{\delta}n.
        $%\end{align*}
    \end{claim}
    \begin{proof}
    	Suppose to the contrary that there exist $i\in \{1,2\}$ and $v\in V_i$ such that $|N_{G}(v) \cap D_i| > 11\sqrt{\delta}n$.
    	Then by the maximality of $G[V_1, V_2]$, we have $|N_{G}(v) \cap V_{3-i}| \ge |N_{G}(v) \cap V_i| \ge |N_{G}(v) \cap D_i| > 11\sqrt{\delta}n$ (since otherwise we can move $v$ from $V_i$ to $V_{3-i}$ and this will result in a bipartition of $G$ with larger number of edges between two parts).
    	By Claim~\ref{CLAIM:GenTuran-cycle-D-bar}, this implies that $|N_{G}(v) \cap D_{3-i}| \ge |N_{G}(v) \cap V_{3-i}|- |\overline{D}|
    	\ge 11\sqrt{\delta}n- 8\delta n
    	> 3\sqrt{\delta} n$, contradicting Claim~\ref{CLAIM:GenTuran-cycle-Maxdeg}.
    	%contradicting Claim~\ref{Appendix:CLAIM:GenTuran-cycle-Maxdeg}.
    \end{proof}
    \begin{claim}[{\cite[Claim~5.10]{LS23a}}]\label{CLAIM:GenTuran-cycle-D'-in-D}
    	We have $|\overline{D'}_i| \le 14\delta n$ for $i\in \{1,2\}$,
    	and $|\overline{D'}| \le 28\delta n$.
    \end{claim}
    \begin{proof}
    	Let $i\in \{1,2\}$.
    	If $|\overline{D'}_i| > 14\delta n$, then
    	$|G[V_1, V_2]| \le |V_1||V_2| - 14\delta n \times ({n}/{2}- \sqrt{2\delta}n - {n}/{3})
    	$ $ < {n^2}/{4} - 2\delta n^2$ by~\eqref{equ:GenTuran-cycle-Vi-size},
    	which contradicts~\eqref{equ:GenTuran-cycle-G[V1,V2]}.
    \end{proof}
    For $i\in \{0, 1,2,3\}$, let $\mathcal{B}_i \coloneqq \left\{K_3\in \mathcal{B} \colon |K_3 \cap \overline{D}| = i\right\}$ and $\mathcal{B}[D]\subseteq \mathcal{B}_0$ with each triangle composed by vertices in $D$.
    It follows~\eqref{equ:upper bound Cktriangle} that
    \begin{align}\label{equ:GenTuran-cycle-B3}
        |\mathcal{B}_3| \le C |\overline{D}|^2.
    \end{align}
    \begin{claim}[{\cite[Claim~5.8]{LS23a}}]\label{CLAIM:GenTuran-cycle-B_2'}
        We have $|\mathcal{B}_2| \le {n |\overline{D}|}/{3}$.
        % \begin{align*}
        %     |\mathcal{B}_2|
        %     \le \frac{n |\overline{D}|}{3}.
        % \end{align*}
    \end{claim}
    \begin{proof}
    	Let us partition $\mathcal{B}_2$ into two sets $\mathcal{B}_2'$ and $\mathcal{B}_2''$, where
    	$$\mathcal{B}_2' \coloneqq \left\{K_3\in \mathcal{B}_2 \colon K_3 \cap D \neq \emptyset\right\}
    	\text{ and }
    	\mathcal{B}_2'' \coloneqq \mathcal{B}_2\setminus \mathcal{B}_2'. $$
    	From the definition of $\mathcal{B}_2''$, each triangle in $\mathcal{B}_2''$ is determined by an edge in links $L_{\overline{D}}(x)$.
    	Then, we have $|\mathcal{B}_2''| \le |\overline{D}|^2.$

    	From the definition of $\mathcal{B}_2'$, each triangle is composed by a vertex in $D$ and two vertices in $\overline{D}$.
    	Let $B_2'$ be the graph spanned by edges in triangles of $\mathcal{B}_2'$,
    	\begin{align*}
    		F \coloneqq E(\overline{D}) \cap E({B}_2'),
    		\text{ }
    		F_1 \coloneqq \left\{e\in F \colon d_{{B}_2'}(e) = 1\right\}
    		\text{ and }
    		F_2 \coloneqq F\setminus F_1.
    	\end{align*}
    	It is clear that $F_1$ contributes at most $|\overline{D}|^{2}$ triangles to $\mathcal{B}_2'$.
    	On the other hand, $F_2$ is a matching in $\overline{D}$ by Claim~\ref{CLAIM:GenTuran-cycle-P_2} (as $N_{B_2'}^{\ast}(e)\in D$ for any $e\in F$).
    	By Claim~\ref{CLAIM:GenTuran-cycle-Maxdeg} and~\eqref{equ:GenTuran-cycle-Vi-size}, we obtain
    	$|\mathcal{B}_2'|
    	\le |\overline{D}|^{2} + \left(\max\{|D_1|,\ |D_2|\}+\sqrt{3\delta}n\right) \times \frac{|F_2|}{2}
    	\le |\overline{D}|^2+ \left(\frac{n}{4}+5\sqrt{\delta}n\right)|\overline{D}|.$
    	
    	Therefore, by Claim~\ref{CLAIM:GenTuran-cycle-D-bar},
    	$$|\mathcal{B}_2| = |\mathcal{B}_2''|+|\mathcal{B}_2'|
    	\le |\overline{D}|^2+ |\overline{D}|^2 + \left(\frac{n}{4}+5\sqrt{\delta}n\right)|\overline{D}|
    	\le \left(16\delta n + \frac{n}{4}+\frac{5}{2}\sqrt{\delta}n\right) |\overline{D}|
    	\le \frac{n |\overline{D}|}{3}.$$
    	This proves Claim~\ref{CLAIM:GenTuran-cycle-B_2'}.
    \end{proof}
Now, we prove several key claims to demonstrate the structural differences compared with the results in~\cite{LS23a}.
    \begin{claim}%[{\cite[Claim~5.9]{LS23a}}]
    	\label{CLAIM:GenTuran-cycle-B1}
        We have $|\mathcal{B}_1| + |\mathcal{B}[D]|
        \le 11\sqrt{\delta} n |\overline{D}| + n-1$.
    \end{claim}

    \begin{proof}
        Let us partition $\mathcal{B}_1$ into two subsets $\mathcal{B}_1'$ and $\mathcal{B}_1''$, where
        \begin{align*}
            \mathcal{B}_1'
            \coloneqq \left\{K_3\in \mathcal{B}_1 \colon x\in K_3\right\}
            \text{ and }
            \mathcal{B}_1''\coloneqq \mathcal{B}_1\setminus \mathcal{B}_1'.
        \end{align*}
        It follows from Claim~\ref{CLAIM:GenTuran-bad-deg-max} that $|\mathcal{B}_1'|
        \le |\overline{D}| \times 11\sqrt{\delta}n
        = 11\sqrt{\delta} n |\overline{D}|.$
        Now we consider triangles in $\mathcal{B}_1''$ and $\mathcal{B}[D]$.
        From the definition of $\mathcal{B}_1''$, each triangle in $\mathcal{B}_1''$ contains exactly one vertex in $\overline{D}$ and two vertices in $D$.
        Observe from Claim~\ref{CLAIM:GenTuran-cycle-P_2}, there will not exist two triangles in $\mathcal{B}_1''\cup \mathcal{B}[D]$ such that those two triangles intersects exactly one vertex.
        Then each triangle in $\mathcal{B}_1''\cup \mathcal{B}[D]$ is contained in vertex disjoint parts, where each part is composed by a copy of $K_4$ or $2$-intersecting graph.
        Since each copy of $K_4$ contains exactly $4$ triangles, by~\eqref{2-intersecting}, we have $|\mathcal{B}_1''| + |\mathcal{B}[D]|\le |D|+ |\overline{D}|$.
        Hence,
        $$|\mathcal{B}_1| + |\mathcal{B}[D]|
        = |\mathcal{B}_1'| + \left(|\mathcal{B}_1''| + |\mathcal{B}[D]| \right)
        \le 11\sqrt{\delta}  n |\overline{D}| + \left(|D|+ |\overline{D}|\right)
        \le 11\sqrt{\delta} n |\overline{D}| + n-1,
        $$
        here we use the fact that $D\cup \overline{D}\subseteq V'$.
        \end{proof}

        \begin{claim}\label{CLAIM:GenTuran-cycle-GD1'D2'}
            We have $\max\{|G[D_1']|,|G[D_2']|\} \le 1$.
        \end{claim}
        \begin{proof}
            Suppose that $uv$ is an edge in $G[D_i']$.
            Then it follows from the definition of $D_i'$ that $\min\left\{|N_{G}(u)\cap V_{3-i}|, |N_{G}(v)\cap V_{3-i}|\right\} \ge 2n/5$.
            Due to~\eqref{equ:GenTuran-cycle-Vi-size} and the Inclusion-exclusion principle, we have
            \begin{align}\label{equ:co-degree}
                |V_{3-i} \cap N_{G}(u) \cap N_{G}(v)|
                \ge 2\times \frac{2}{5}n - \left(\frac{n}{2}+\sqrt{2\delta}n\right)
                = \frac{3}{10}n-\sqrt{2\delta}n.
            \end{align}
            Suppose that there exist two distinct edges $e_1', e_2' \in E(G[D_1']) \subseteq E(G[D_1])$.
            If $e_1'$ intersects $e_2'$ at $v$, then $L_G(v)$ would contradict Claim~\ref{CLAIM:GenTuran-cycle-P_2} by~\eqref{equ:co-degree}.
            Otherwise, it follows from~\eqref{equ:co-degree} and the Inclusion-exclusion principle again that $$|N_{G}^{\ast}(e_1')\cap N_{G}^{\ast}(e_2')\cap V_{3-i}|\ge 2\times \left(\frac{3}{10}n - \sqrt{2\delta}n\right) - \left(\frac{n}{2} + \sqrt{2\delta}n\right)
            = \frac{n}{10} - 3\times\sqrt{2\delta}n
            \ge \frac{n}{20},$$
            which would contradict Claim~\ref{CLAIM:GenTuran-cycle-P_2}.
        \end{proof}
        Observe from the proof of Claim~\ref{CLAIM:GenTuran-cycle-GD1'D2'}, it can be verified that
        \begin{equation}\label{equ:GenTuran-Di-triangle-free}
        \mbox{both $G[D_1']$ and $G[D_2']$ contain no triangle.}
        \end{equation}
        Moreover, we can obtain the following
        \begin{claim}\label{GenTuran-no-D1,D2-edge}
        	If $e_1$ is an edge in $G[D_1']$ and $e_2$ is an edge in $G[D_2']$, then there is no triangle composed by $e_i$ and a vertex in $e_{3-i}$ for $i\in \{1,2\}$.
        \end{claim}
        \begin{proof}
        	Let $e_1=u_1v_1\subseteq E(G[D_1'])$, $e_2=u_2v_2\subseteq E(G[D_2'])$ and $v_1u_1u_2v_1$ be a triangle.
        	By~\eqref{equ:co-degree}, we can choose $w\in N^{\ast}_{V_1}(u_2v_2)\setminus \{u_1v_1\}$.
        	Clearly, $u_1v_1, v_2w\in N_{V'}(u_2)$ will contradictis Claim~\ref{CLAIM:GenTuran-cycle-P_2}.
        \end{proof}
        %%%%%%%%%%%%%%%%%%%%%%%%%%%%%%%%%%%%%%%%
        \iffalse
        \begin{claim}\label{CLAIM:GenTuran-cycle-2-intersecting}
        We have $N(K_3,{B}[D]) \le n-3$.
        %In particular, $N(K_3,{B}[D]) \le (1/2 + \sqrt{2\delta}n$, if ${\cal B}[V_1]=\emptyset$ or ${\cal B}[V_2]=\emptyset$.
        \end{claim}
        \begin{proof}
        Let ${\cal B}[D]$ be the triangles in ${\cal B}$ such that each triangle is composed by vertices in $D$.
        First, observe from Claim~\ref{CLAIM:GenTuran-cycle-P_2} that ${\cal B}[D]$ is composed by disjoint maximal $2$-intersectings and each triangle in ${\cal B}[D]$ is contained in exactly one $2$-intersecting.
        Suppose ${\cal B}[D]$ contains $s$ disjoint maximal $2$-intersectings, denoted by $H_1, H_2, \dots, H_s$, and let $e_i$ be the intersecting edge in $H_i$.
        Clearly, $N(K_3,{B}[D]) \le n-1-2i\le n-3$.
        \end{proof}
        \fi
        %
        \begin{claim}\label{CLAIM:GenTuran-even-cycle-B0}
            We have $|\mathcal{B}_0\setminus\mathcal{B}[D]|
            \le 11\sqrt{\delta} n |\overline{D'}| + |E(G[D_1'])| + |E(G[D_2'])|$.
        \end{claim}
        \begin{proof}
            It follows from the definition that every triangle in $\mathcal{B}_0\setminus\mathcal{B}[D]$ is determined by an edge in links $L_{D_i}(x)$ for some $i\in \{1,2\}$.
            By Claim~\ref{CLAIM:GenTuran-bad-deg-max}, the number of triangles in $\mathcal{B}_0\setminus\mathcal{B}[D]$ that have nonempty intersection with $\overline{D'}$ is at most $11\sqrt{\delta}n \left(|\overline{D'}_1| + |\overline{D'}_2|\right) = 11\sqrt{\delta} n |\overline{D'}|$.
            On the other hand, following from the definition of ${\cal B}$ and Claim~\ref{CLAIM:GenTuran-cycle-GD1'D2'}, the number of triangles in $\mathcal{B}_0\setminus\mathcal{B}[D]$ that have empty intersection with $\overline{D'}$ depends on the number of edges in $G[D_1']$ and $G[D_2']$.
        \end{proof}
    Recall the definition of ${\cal M}_2$.
    Since every vertex $v\in \overline{D'}_i$ has at most $\tau'$ neighbors in $D_{3-i}$ for $i=1,2$, each edge $vx$ contributes at least $\left(|D_{3-i}| -\tau'\right)$ triangles in $\mathcal{M}_2$.
    Combined with~\eqref{equ:GenTuran-Di-size}, we see that $vx$ contributes at least $\left({1}/{10} - 2\sqrt{\delta} \right)n$ triangles in $\mathcal{M}_2$ for each $v\in \overline{D'}$,
    and then
    \begin{equation}\label{equ:GenTuran-D''-M2}
    	|\mathcal{M}_2| \ge \left(\frac{1}{10} - 2\sqrt{\delta} \right)n|\overline{D'}|
    \end{equation}
Now, we count the triangles in ${\cal S}\cup {\cal B}\setminus {\cal M}$.

If either $\overline{D}\ne \emptyset$ or $\overline{D'}\ne \emptyset$, then it follows from%~\eqref{equ:GenTuran-B[D]},
~\eqref{equ:GenTuran-cycle-B3},~\eqref{equ:GenTuran-D''-M2} and Claims~\ref{CLAIM:GenTuran-cycle-D-bar},~\ref{CLAIM:GenTuran-cycle-B_2'},~\ref{CLAIM:GenTuran-cycle-B1},~\ref{CLAIM:GenTuran-even-cycle-B0} that
\begin{align*}
    N(K_3,G)
    \le &\ |\mathcal{S}| + |\mathcal{B}_3| + |\mathcal{B}_2|
    + |\mathcal{B}_1| + |\mathcal{B}[D]| +  |\mathcal{B}_0\setminus\mathcal{B}[D]| - |\mathcal{M}_1| - |\mathcal{M}_2| \\
    \le &\ N(K_3,S(n)) + C |\overline{D}|^2
    + \frac{n |\overline{D}|}{3} + \left(11\sqrt{\delta} n |\overline{D}| + n-1\right) + \left(11\sqrt{\delta} n |\overline{D'}| + 2 \right) \\
    &\ - \frac{49}{100}  n|\overline{D}|
    - \left(\frac{1}{10} - 2\sqrt{\delta} \right)n|\overline{D'}|
    \\ %%%%%%%%%%
    \le &\ N(K_3,S(n)) + (n + 1) - \left(\frac{47}{300} - 11\sqrt{\delta} - 8\delta C \right)
    n|\overline{D}| - \left(\frac{1}{10} - 13\sqrt{\delta} \right)n|\overline{D'}| \\
    < &\ N(K_3,S(n)) + (n-7) = N(K_3,S^=(n)),
\end{align*}
as $n$ is sufficiently large and $\delta$ is sufficiently small,
a contradiction.

Therefore, we suppose $\overline{D} = \overline{D'} = \emptyset$.
Then $D_1'=D_1$, $D_2'=D_2$, $D=V'$ and ${\cal B}_1 = {\cal B}_2 = {\cal B}_3 = {\cal M}_1 = \emptyset$.
By~\eqref{equ:GenTuran-Di-triangle-free}, we see that $G[D_1]$ and $G[D_2]$ contain no triangle, implying that the triangles in ${\cal B}[D]$ are composed by a fixed edge in $G[D_1]$ and a vertex in $D_{2}\cup \{x\}$ or a fixed edge in $G[D_2]$ and a vertex in $D_{1}\cup \{x\}$.
From Claim~\ref{CLAIM:GenTuran-cycle-GD1'D2'}, it remains to consider the following conditions:

If both $G[D_1]$ and $G[D_2]$ contain no edge, then it is not hard to verify that the graph $G$, which contains maximum number of triangles, is exactly $S(n)$, contradicting~\eqref{equ:sequence S}.

If $|E(G[D_1])| + |E(G[D_2])| = 1$, then $G$ must be $S^+(n)$, contradicting~\eqref{equ:sequence S}.

If $|E(G[D_1])| = 1$ and $|E(G[D_2])| = 1$, then $G$ is exactly $S^=(n)$ by Claims~\ref{GenTuran-no-D1,D2-edge}.
\qed
%\cite{MR4299397}

\section*{Declarations}
 \subsection*{Funding}
 
This work is  supported by Science and Technology Commission of Shanghai Municipality (No. 22DZ2229014) and National Natural Science Foundation of China (No. 12331014, No. 12501481, No. 12271169).

 %\subsection*{Author Contributions} 
 
% \textbf{Jialei Song:} Conceptualization, Methodology, Formal analysis, Writing-original draft.  \textbf{Qi Wu:} Validation, Writing-review \& editing. \textbf{Long-Tu Yuan:} Conceptualization, Supervision, Funding acquisition, Writing-review \& editing.
 
 \subsection*{Data Availability Statement} 
 Data sharing is not applicable to this article as no new data were created or analyzed in this study.

 \subsection*{Conflict of Interest/Competing Interests} 
 The authors have no relevant financial or non-financial interests to disclose.

\end{document}